\documentclass[12pt,reqno]{amsart}

\usepackage[usenames]{color}

\usepackage{amsmath,epsfig}

\newtheorem{thm}{Theorem}[section]
\newtheorem{cor}[thm]{Corollary}
\newtheorem{prop}[thm]{Proposition}
\newtheorem{lemma}[thm]{Lemma}
\newtheorem{defn}[thm]{Definition}
\newtheorem{rem}[thm]{Remark}

\newtheorem{exa}[thm]{Example}
\newtheorem{question}[thm]{Question}

 \newcommand{\R}{\mathbb R}

\numberwithin{equation}{section}
\numberwithin{table}{section} 
\numberwithin{figure}{section}

\begin{document}

\title{Isoperimetric control of the Steklov spectrum}
\author{Bruno Colbois}
\address{Universit\'e de Neuch\^atel, Institut de Math\'ematiques, Rue Emile-Argand 11, Case postale 158, 2009 Neuch\^atel
Switzerland}
\email{bruno.colbois@unine.ch}
\author{Ahmad El Soufi}
\address{Laboratoire de Math\'{e}matiques et Physique Th\'{e}orique, UMR-CNRS 6083, Universit\'{e} François Rabelais de Tours, Parc de Grandmont, 37200 Tours, France}
\email{elsoufi@univ-tours.fr}
\author{Alexandre Girouard}
\address{Universit\'e de Neuch\^atel, Institut de Math\'ematiques, Rue Emile-Argand 11, Case postale 158, 2009 Neuch\^atel
Switzerland }
\email{alexandre.girouard@unine.ch}

\date{\today}
\begin{abstract}
  Let $(N,g)$ be a complete Riemannian manifold of dimension $n+1$
  whose Riemannian metric $g$ is conformally equivalent to a metric
  with non-negative Ricci curvature. 
  The normalized Steklov eigenvalues $\overline{\sigma}_k(\Omega)$ of
  a bounded domain $\Omega$ in $N$ are bounded above in terms of the
  isoperimetric ratio of the domain. Consequently, the normalized Steklov
  eigenvalues of a bounded domain $\Omega$ in Euclidean space,
  hyperbolic space or a standard hemisphere are uniformly bounded
  above~:
  $\overline{\sigma}_k(\Omega)\leq C(n) k^{2/(n+1)},$
  where $C(n)$ is a constant depending only on the dimension.
  On a compact surface $\Sigma$ with
  boundary, the normalized Steklov eigenvalues are uniformly bounded
  above in terms of genus~:
  $\overline{\sigma}_k(\Sigma)\leq C\left(1+\mbox{genus}(\Sigma)\right)k.$
  We also obtain a relationship between the Steklov eigenvalues of a
  domain $\Omega$ and the eigenvalues of the Laplace-Beltrami operator
  on the hypersurface bounding $\Omega$. 
\end{abstract}




\maketitle

\section{Introduction}
The goal of this paper is to obtain geometric upper bounds for the
spectrum of the Dirichlet-to-Neumann map.
Let $N$ be a complete Riemannian manifold. Let $\Omega$ be a relatively compact
domain in $N$ with smooth boundary $\Sigma$. The Dirichlet-to-Neumann
map $\Lambda: C^{\infty}(\Sigma)\rightarrow C^{\infty}(\Sigma)$ is
defined by 
$$\Lambda f=\partial_n(Hf)$$
where $Hf$ is the harmonic extension of $f$ to the interior of $\Omega$ and
$\partial_n$ is the outward normal derivative.
The Dirichlet-to-Neumann map is a first order elliptic
pseudodifferential operator~\cite{taylorPDEII}. Because $\Sigma$ is
compact, the spectrum of $\Lambda$ is positive, discrete and
unbounded~\cite[p. 95]{band}: 
\begin{gather*}
  0 = \sigma_1\le \sigma_2(\Omega) \leq  \sigma_3(\Omega) \leq \cdots
  \nearrow\infty.
\end{gather*}
The spectrum of this operator is also called the Steklov spectrum
of the domain $\Omega$.

\subsection{Physical interpretation}
Prototypical in inverse problems, the Dirichlet-to-Neumann map is
closely related to the Calder\'on problem~\cite{cldr} of determining
the anisotropic conductivity of a body from current and voltage
measurements at its boundary. This point of view makes it useful as a
model for Electrical Impedance Tomography. A particularly striking
related result~\cite{LaTaUh} is that if the manifold $M$ is real
analytic of dimension at least 3, then the knowledge of $\Lambda$
determines $M$ up to isometry. The study of the spectrum of $\Lambda$
was initiated by Steklov in 1902~\cite{stek}. Eigenvalues and
eigenfunctions of this operator are used in fluid mechanics, heat
transmission and vibration problems~\cite{foxkut,kokrein}.

\subsection{Optimization}
The general question we are interested in is to give upper bounds for
the eigenvalues in terms of natural geometric quantities. Because the
eigenvalues are not invariant under scaling of the Riemannian metric,
we consider normalized eigenvalues 
$$\bar\sigma_k(\Omega):=\sigma_k(\Omega)|\Sigma|^{\frac{1}{n}}
\quad\mbox{ with }
\quad|\Sigma|=\int_{\Sigma}dv_{\Sigma}$$
where $n$ is the dimension of the boundary $\Sigma$ and $dv_{\Sigma}$
is the measure induced by the Riemannian metric of $N$ restricted to
$\Sigma$.


\medskip
\begin{question}
  Given a complete Riemannian manifold $N$, is
  $\overline{\sigma}_k(\Omega)$ uniformly bounded above among bounded
  domains $\Omega\subset N$?
\end{question}

\medskip
For the first non-zero eigenvalue, this question has been studied
by many
authors. See~\cite{payne,hps} for early results in the planar
case. The series of paper by J. Escobar~\cite{esco1,esco2,esco3} is 
influencial. For more recent results,
see~\cite{wangxia2,fraschoen}. For higher eigenvalues in the planar
situation, see~\cite{gp,gp2}.

\medskip
The main result of this paper (Theorem~\ref{maindegree}) is an upper
bound for the eigenvalues of the Dirichlet-to-Neumann map on a domain
in a complete Riemannian manifold satisfying a growth and a packing
condition in terms of its isoperimetric ratio. We list here some
applications.

\subsubsection*{Domains in space forms}
The case of simply connected planar domain is well
understood. See~\cite{wein, hps} and especially~\cite{gp2} for a
survey of this problem. If a bounded domain $\Omega \subset \R^2$ is
simply connected, then 
\begin{gather}\label{ineqPlanarDomain}
  \bar\sigma_k(\Omega)\le 2\pi k.
\end{gather}
This inequality is optimal.
In higher dimensions, only few results
about the first non-zero eigenvalue are known. See~\cite{brock} for a
different normalization.

\medskip
Our first result is a generalization of the above to the case of
arbitrary\footnote{i.e. not necessarily simply connected.} domains in
space forms.
\begin{thm}\label{thmSpaceform}
  There exists a constant $C_n$ depending only on the dimension $n$
  such that, for each bounded domain $\Omega$ in a space form
  $\R^n$, $\mathbb H^n$ or in an hemisphere of $\mathbb S^n$, we
  have 
  \begin{gather}\label{ineqSpaceForm}
    \bar\sigma_k(\Omega)\leq C_n k^{2/n}.
  \end{gather}
\end{thm}

This result follows from a more general result allowing control of
the Steklov spectrum of a domain in a complete manifold in terms of
its isoperimetric ratio.

\subsubsection*{Domains in a complete manifold}
The following theorem shows that under an additional assumption on
Ricci curvature, we can control the normalized Steklov eigenvalue
$\overline{\sigma}_k$ of a domain in terms of its isoperimetric ratio.

\medskip
\begin{thm}\label{ricpos}
  Let $N$ be a complete  manifold of dimension $n+1$. If $N$ is
  conformally equivalent to a complete manifold 
  with non-negative Ricci curvature,  then for each domain
  $\Omega\subset N$, we have 
  \begin{gather}\label{ineqricpos}
    \bar\sigma_k(\Omega) \le \frac{\gamma(n)}{I(\Omega)^{\frac{n-1}{n}}}k^{2/(n+1)},
  \end{gather}
  where $I(\Omega)$ is the classical isoperimetric ratio of 
  $\Omega$, namely
  \begin{gather*} 
    I(\Omega) = \frac{\vert \Sigma \vert}{\vert \Omega \vert^{n/(n+1)}}.
  \end{gather*}
\end{thm}

\medskip
A surprising corollary of this theorem is that if
$\dim N\geq 3$ then a large isoperimetric ratio $I(\Omega)$  
implies that the normalized eigenvalue $\bar\sigma_k(\Omega)$ is
small. This is false for surfaces ($n=1$), see
Example~\ref{exampleSurface}.

\begin{rem}
  Since there exists a constant $c_n$ such that
  $$\overline{\sigma}_k(\Omega)\sim c_n k^{1/n}\mbox{ as }k\rightarrow\infty,$$
  one may expect that a bound such as~(\ref{ineqricpos}) should hold
  with exponents $1/n$. In fact, for $n\geq 2$, this is impossible
  because it would imply an upper bound on $I(\Omega)$. Naturally, if
  we remove $I(\Omega)$, such a bound might still be possible. For
  instance, we do not know if inequality~(\ref{ineqSpaceForm}) holds
  with exponent improved to $1/n$.
\end{rem}

\subsubsection*{Large eigenvalues}
The assumption of non-negative Ricci curvature is essential.
In section~\ref{large} we will construct for each $n\geq 2$ and each
$\kappa<0$ a complete manifold $N$ of dimension $n+1$ with Ricci
curvature bounded below by $\kappa$ admitting a sequence $\Omega_j$ of
domains such that the normalized eigenvalues
$\bar\sigma_2(\Omega_j)\to\infty$ and the isoperimetric ratio
$I(\Omega_j)\rightarrow\infty$.

\medskip
Under the assumption of non-negative Ricci curvature, we do not
know if the presence of the isoperimetric ratio is
essential. Namely, is there a constant $C(n,k)$ such that for each
domain $\Omega\subset N$, $\bar\sigma_k(\Omega) \le C(n,k)$ ?
Of course, this will be the case if we can give uniform lower bound on
the isoperimetric ratio $I(\Omega)$. This situation will be discussed
in Proposition~\ref{croke} and in Corollary~\ref{domric}.

\subsection{Surfaces}

If $N$ is two-dimensional the isoperimetric ratio disappears from
inequality~(\ref{ineqricpos}). This means that for any domain in a
complete surface with conformally non-negative curvature we get a
uniform bound similar to~(\ref{ineqPlanarDomain})~:
\begin{gather*}
  \bar\sigma_k(\Omega) \le \gamma(2)k.
\end{gather*}

In fact, in the case of surfaces, we don't need to assume our compact
manifold to be a domain in a complete manifold with non-negative
  Ricci curvature. Let $M$ be a compact surface with smooth
  boundary $\Sigma$. The Steklov spectrum of $M$ 
  is defined exactly as in the case of a domain.
\begin{thm}\label{surfaces}
  There exists a constant $C$ such that for any compact orientable
  Riemannian surface $M$ of genus $\gamma$ with non-empty smooth
  boundary,
  \begin{align}
    \bar\sigma_k(M)\leq C\left\lfloor\frac{\gamma+3}{2}\right\rfloor k\label{ineqsurfaces}
  \end{align}
  where $\lfloor.\rfloor$ is the integer part.
\end{thm}

\medskip
This result is in the spirit of Korevaar~\cite{kvr} and generalizes
a recent result of Fraser-Schoen~\cite{fraschoen}.

\subsection{Relationships with the spectrum of the Laplacian for
  Euclidean hypersurface}
In section~\ref{sectionrelations}, we will use the result of our
paper~\cite{ceg1} to establish a relation between the spectrum of the
Dirichlet-to-Neumann map and the spectrum of the Laplacian acting on
smooth function of the boundary $\Sigma$. The main consequence of this
estimate is that for a manifold embedded as hypersurface in Euclidean
space, the presence of large normalized eigenvalue of the
Laplacian will force
the normalized eigenvalues $\overline{\sigma}_k$ to be small.


\section{Statement and proof of the main theorem }\label{sectionmainthm}
We consider a slightly more general eigenvalue problem than that
of the introduction. Let $M$ be a sufficiently regular compact
Riemannian manifold of dimension $n+1$ with boundary $\Sigma$.  Let
$\delta$ is a smooth non-negative and non identically zero function on
$\Sigma$.
The Steklov eigenvalue problem is
\begin{gather*}
  \Delta f=0 \ \mbox{ in } M,\\
  \partial_nf=\sigma\, \delta f \mbox{ on } \Sigma.
\end{gather*}
It has positive and discrete spectrum~\cite[p. 95]{band}:
\begin{gather*}
  0 = \sigma_1\le \sigma_2(M,\delta) \leq \sigma_3(M,\delta) \leq \cdots \nearrow\infty.
\end{gather*}
Because the eigenvalues are not
invariant under scaling of the Riemannian metric or of the mass
density $\delta$, we consider the normalized eigenvalues 
$$\bar\sigma_k(M,\delta):=\sigma_k(M,\delta) {m(\Sigma,\delta)} |\Sigma|^{\frac 1 n},$$ with 
$$ |\Sigma|=\int_{\Sigma}dv_{\Sigma} \quad\mbox{ and } \quad m(\Sigma,\delta)
=\frac 1 {|\Sigma|}\int_{\Sigma}\delta\,dv_{\Sigma}.$$

\medskip
Let $(N,g_0)$ be a complete Riemannian manifold of dimension
$(n+1)$. We consider the Riemannian distance $d_0$ induced by $g_0$
and we assume:

\medskip
\noindent
(\textbf{P1}) There exists a constant $C$ depending on $d_0$ such that
each ball of radius $2r$ in $N_0$ may be covered by at most $C$ balls of radius $r$.

\medskip
\noindent
(\textbf{P2})
There exists a constant $\omega$ depending only on $g_0$ such that, for each $x\in N_0$,
and $r\ge 0$, $\vert B(x,r)\vert \le \omega r^{n+1}$.

\medskip
\begin{exa} \label{control}
There is a large supply of complete Riemannian manifolds satisfying
these conditions.
\begin{enumerate}
\item
  If $N$ is compact, then (P1) and (P2) are satisfied. In this
  case the constants $C$ and $\omega$ depend on $g_0$.
\item
  If the Ricci curvature of $g_0$ is non-negative then, by
  Bishop-Gromov comparison theorem, there exist constants $C$ and
  $\omega$ depending only on the dimension of $N$
  such that (P1) and (P2) are satisfied. This is in particular
  the case of the Euclidean space $\mathbb R^{n+1}$, and we will use
  this in the proof of Theorem~\ref{spaceforms} and
  Theorem~\ref{ricpos}.
\end{enumerate}
\end{exa}

\medskip
It follows from the previous example that Theorem~\ref{ricpos} is a
corollary of the following theorem.
\begin{thm}\label{maindegree}
  Let $(N,g_0)$ be a complete Riemannian manifold of dimension $(n+1)$ 
  satisfying $(P1)$ and $(P2)$. Let $g\in[g_0]$ be a metric in the
  conformal class of $g_0$.
  Then, there exists a constant $\gamma(g_0)$ depending only on the
  constants $C$ and $\omega$ coming from $(P1)$ and $(P2)$ such that,
  for any bounded domain $\Omega\subset N$ and any density $\delta$ on
  $\Sigma=\partial\Omega$, we have 
  \begin{gather}\label{ineqmainthm}
    \bar\sigma_k(\Omega,\delta)
    \leq\frac{\gamma(g_0)}{I(\Omega)^{\frac{n-1}{n}}}k^{2/(n+1)}.
  \end{gather}
\end{thm}

\medskip
The proof of Theorem~\ref{maindegree} is based on the
construction of a family of disjointly supported functions with
controlled Rayleigh quotient 

\begin{gather*}
  R(f)=
  \frac{\int_{\Omega}\vert \nabla_{g} f \vert^2 dv_g}{\int_{\Sigma}f^2\delta dv_{\Sigma}}.
\end{gather*}

\medskip
On $N$ we consider the Borel measure $\mu=\delta dv_{\Sigma}$.
That is, the measure of an open set $\mathcal{O}\subset N$ is
\begin{align}\label{measure}
  \mu(\mathcal{O})=\int_{\mathcal{O}\cap\Sigma}\delta\,dv_{\Sigma}.
\end{align}

\medskip
In particular, we have
$$
\mu(N)= \int_{\Sigma}\delta dv_{\Sigma}= \vert \Sigma \vert m(\Sigma,\delta).
$$

\begin{defn} Let $(X,d)$ be a metric space.
  An \emph{annulus} $A\subset X$ is a subset of the form $\{x \in X:
  r<d(x,a)<R\}$ where $a\in X$ and $0 \le r < R<\infty$. The annulus
  $2A$ is the annulus $\{x \in X: r/2 < d(x,a)< 2R\}$. In particular,
  $A \subset 2A$.
\end{defn}

\medskip
Theorem 1.1 and Corollary 3.12 of~\cite{gyn} tell us that if a metric
measured space $(X,d,\nu)$ satisfy property $(P1)$ and if the measure
$\nu$ is non-atomic, then there is a constant $c>0$ such that, for
each positive integer $k$, there exist a family of $2k$ annuli
$\{A_i\}_{i=1}^{2k}$ in $X$ such that
$$
\mu(A_i) \ge c\frac{\nu(X)}{k}.
$$
and the annuli $2A_i$ are disjoint.

\medskip
The constant $c$ depends only on the constant $C$ of property $(P1)$,
that is only on the distance $d$ and not on the measure $\nu$.

\begin{proof}[Proof of Theorem~\ref{maindegree}] 
  Consider the metric measured space $(N,d_0,\mu)$, where $d_0$ is
  the Riemannian distance associated to $g_0$ and $\mu$ is the measure
  induced by $g_0$ and the density $\delta$ as defined above in
  (\ref{measure}).
  
  \smallskip
  It follows from Theorem~1 and Corollary~3.12 of~\cite{gyn}
  mentioned above that there exist $2k$ annuli
  $A_1,...,A_{2k} \subset N$ with
  
  \begin{gather}\label{IneqMuAibounded}
    \mu(A_i) \ge \frac{\mu(N)}{ck},\quad
    c=c(g_0)>0.
  \end{gather}
  The annuli $B_i=2A_i$ are mutually disjoint. We can reorder them so
  that the first $k$ of them satisfy
  
  \begin{gather}\label{firstk}
    \vert B_i \cap \Omega \vert_g
    \leq\frac{\vert \Omega \vert_g }{k}\ \ \ (i=1,\cdots,k).
  \end{gather}

  \medskip
  Let $A=\{x \in N: r<d(x,a)<R\}$ be one of these first $k$ annuli and
  let $h$ a function supported in $2A$.
  Taking~(\ref{firstk}) into account, it follows from H\"older's
  inequality and the conformal invariance of the generalized Dirichlet
  energy that
  \begin{align*}
    \int_{B \cap \Omega} \vert \nabla_g h\vert^2\,dv_g&\leq
    \left(\int_{B \cap \Omega} \vert \nabla_g h\vert^{n+1}\,dv_g\right)^{2/(n+1)}
    \vert B \cap \Omega\vert_g^{1-2/(n+1)}\\
    &\leq
     \left(\int_{2A} \vert \nabla_{g_0} h \vert^{n+1}\,dv_{g_0}\right)^{2/(n+1)}
     \left(\frac{\vert \Omega \vert_g}{k}\right)^{1-2/(n+1)}
  \end{align*}
  
  \medskip
  Choosing the function $h$ that is identically $1$ on $A$ and
  proportional to the distance to $A$ on $2A \setminus A$, we have  
  \begin{gather*}
    \vert \nabla_{g_0} h \vert^{n+1}\leq
    \begin{cases}
      \frac{2^{n+1}}{r^{n+1}}& \mbox{ on } B(a,r)\setminus B(a,r/2),\\
      \frac{1}{R^{n+1}}& \mbox{ on } B(a,2R)\setminus B(a,R).
    \end{cases}
  \end{gather*}
  
  \medskip
  It follows from $(P_2)$ that
  \begin{gather*}
    \int_{2A} \vert \nabla_{g_0} h \vert^{n+1}\,dv_{g_0}
    \leq 2^{n+2}\omega.
  \end{gather*}
  
  This leads to
  \begin{gather*}
    \int_{B \cap \Omega} \vert \nabla_g h\vert^2\,dv_g
    \leq (2^{n+2}\omega)^{2/(n+1)}\left(\frac{\vert \Omega \vert_g }{k}\right)^{(n-1)/(n+1)}
  \end{gather*}

Moreover, using~\eqref{IneqMuAibounded} we get

\begin{gather*}
  \int_{\Sigma} h^2\,\delta dv_{\Sigma} \ge \mu(A) \ge  \frac{\mu(N)}{ck}
\end{gather*}

\medskip
By considering the Rayleigh quotient, this leads to

\begin{gather*}
  \sigma_k(\Omega,\delta)\leq
  \frac{2^{n+2}\omega ck}{\mu(N)}
  \left(\frac{\vert \Omega \vert_g}{k}\right)^{(n-1)/(n+1)}.
\end{gather*}

\smallskip
Using $\mu(N)= \vert \Sigma \vert m(\Sigma,\delta)$,
we conclude

\begin{gather*}
  \bar\sigma_k(\Omega,\delta)= 
  \sigma_k(\Omega,\delta)m(\Sigma,\delta)\vert \Sigma \vert^{1/n}
  \leq
  \frac{\gamma(g_0)}{I(\Omega)^{\frac{n-1}{n}}}k^{2/(n+1)},
\end{gather*}

with $\gamma(g_0)=2^{n+2}c \omega$.
\end{proof}

\section{Applications of Theorem \ref{maindegree}}  \label{applications}

In this section, we prove most of the results announced in the
introduction as consequence of our Theorem~\ref{maindegree}.

\subsection{Domains in a manifold with conformally non-negative Ricci curvature}
It is difficult to estimate the packing constant $C$ and the growth
constant $\omega$ of a general Riemannian manifold. Nevertheless, as
was observed in Example~\ref{control}, in the special situation where
$\Omega$ is a domain $\Omega$ in a complete Riemannian manifold $N$
with non-negative Ricci curvature, it follows from the Bishop-Gromov
inequality that these constants can be estimated in terms of the
dimension.

\begin{thm}\label{ricposDensity}
  Let $(N,g)$ be a complete Riemannian manifold of dimension $(n+1)$ and assume that
  the metric $g$ is conformally equivalent to a metric $g_0$ with $Ric(g_0) \ge 0$. Then,
  for any bounded domain $\Omega \subset N$, and for any density $\delta$ on $\partial \Omega$,
  we have
  
  \begin{align}
    \bar\sigma_k(\Omega,\delta)
    &\leq\frac{\gamma(n)}{I(\Omega)^{\frac{n-1}{n}}}k^{2/(n+1)},\label{thmSteklovGeneral}
  \end{align}
  where $\gamma(n)$ is a constant depending only on $n$.
\end{thm}

This theorem is a direct consequence of Theorem \ref{maindegree} and
of Example~\ref{control}. Theorem~\ref{ricpos} is the special case
when $\delta\equiv 1$.

\medskip
If $n \ge 2$,
large isoperimetric ratio $I(\Omega)$ implies small eigenvalues
$\bar\sigma_k(\Omega,\delta)$.
\begin{cor}\label{corLargeIso}
  Under the assumptions of Theorem \ref{ricposDensity}, if a family of
  domains $\{\Omega_t\}_{0<t<1}$ is such that
  $\displaystyle\lim_{t\rightarrow 0}I(\Omega_t)=\infty,$ then, if $n
  \ge 2$ and for each density $\delta_t$ 
  on $\partial \Omega_t$, we have
  $$\lim_{t\rightarrow 0}\bar\sigma_k(\Omega_t,\delta_t)\rightarrow 0.$$
\end{cor}
This is false for $n=2$. See Example~\ref{exampleSurface}.

\subsection{Control of the isoperimetric ratio}
In general, it is difficult to estimate the isoperimetric ratio
$I(\Omega)$. We give two special situations where we have a
uniform lower estimate on it. This will be a consequence of
the inequality of Croke~\cite{croke} as presented by
Chavel~\cite[p.136]{cha2}.

\begin{prop}\label{croke}
  Let $N$ be a complete Riemannian manifold. For each $x \in N$, let
  $\mbox{inj}(x)$ the injectivity radius of $N$ at $x$.
  Given $p \in N$ and $\rho >0$, consider
  \begin{gather}\label{rCroke}
    r < \frac{1}{2}\left(inf_{x \in B(p,\rho)}\mbox{inj}(x)\right).
  \end{gather}
  
  \smallskip
  Then, for a each domain $\Omega \subset B(p,r)$, we have
  $I(\Omega)\ge C(n)$  
  for a constant $C(n)$ depending only on the dimension.
\end{prop}

If the injectivity radius of $N$ is strictly positive,
we can choose any $r<\frac{\mbox{inj}(N)}{2}$.

\subsubsection{Domains in space forms}
A special but very important case is when the ambient space $N$ is
a space form, that is the Euclidean space $\mathbb R^{n+1}$, the
hyperbolic space $\mathbb H^{n+1}$ or the sphere $\mathbb S^{n+1}$
with their natural metric of curvature $0,-1$ and $1$ respectively. 

\begin{thm}\label{spaceforms}
  For any bounded domain $\Omega$ with smooth boundary
  $\Sigma=\partial\Omega$ in $\R^{n+1}$,
  $\mathbb H^{n+1}$ or on an hemisphere of the sphere $\mathbb{S}^{n+1}$  and any
  $k\ge 1$, we have
  $$ 
  \bar\sigma_k(\Omega,\delta) \leq
  \frac{\gamma(n)}{I(\Omega)^{\frac{n-1}{n}}}k^{2/(n+1)}
  \leq C_nk^{2/(n+1)}
  $$
  where $C_n$ and $\gamma_n$ are constants depending only on $n$. 
\end{thm}

\begin{proof}
  The standard metrics on Euclidean space and on the sphere have
  non-negative Ricci curvature. The standard metric on the hyperbolic
  space is conformally equivalent to the Euclidean one. We can
  therefore apply Theorem~\ref{maindegree}, with $g_0=g$ one of these
  standard metric.
  
  The injectivity radii of Euclidean and hyperbolic space are
  infinite. That of the unit sphere is $\pi$. The proof is 
  completed by using Proposition~\ref{croke}.
\end{proof}

In particular, this proves Theorem~\ref{thmSpaceform}.

\begin{rem}
  It is also classically known that any domain $\Omega$ in Euclidean
  space, the hyperbolic space or an hemisphere, isoperimetric ratio
  bounded from   below by a constant depending on the dimension. This
  can be used instead of Croke's result in the above proof.
\end{rem}

\subsubsection{Domains inside a ball} 
In the case where the Ricci curvature of $N$ is non-negative, we
deduce the following 

\begin{cor} \label{domric}
  If the Ricci curvature of $N$ is non-negative and if
  $\Omega \subset B(p,r)$, where $r$ satisfy~\eqref{rCroke}, then
  \begin{gather*}
    \bar\sigma_k(\Omega,\delta) \leq C_n k^{2/(n+1)}
  \end{gather*}
  for some constant $C_n$ depending only on $n$.
\end{cor}

\section{Relation between the spectrum of the Dirichlet-to-Neumann
  operator and the spectrum of the Laplacian. }\label{sectionrelations}
Let $\Delta_\Sigma$ be the Laplacian acting on smooth functions of the
boundary $\Sigma=\partial M$ of a compact Riemannian manifold with
boundary. Let
$0=\lambda_1\leq\lambda_2(\Sigma)\leq\cdots\nearrow\infty$
be the spectrum of $\Delta_\Sigma$.
It is well known that $\Lambda$ is a first order
pseudodifferential operator and that its principal symbol is the square root of the
principal symbol of $\Delta_\Sigma$.
It follows that $\sigma_k\sim\sqrt{\lambda_k}$ as
$k\rightarrow\infty$. See for instance~(\cite[p. 38 and
p. 453]{taylorPDEII}, ~\cite{shamma}).

\begin{question}
  Can the eigenvalues $\sigma_k$ and $\lambda_l$ be compared to each
  other ?
\end{question}

Recently, Wang and Xia studied this question~\cite{wangxia2} for the
first non-zero eigenvalues of both operators. Under the assumption
that Ricci curvature of $M$ is non-negative and that the principal 
curvatures of $\partial M$ are bounded below by a positive constant
$c$, they proved that 
\begin{align*}
  \sigma_2 \le \frac{\sqrt{\lambda_2}}{nc}(\sqrt{\lambda_2}+\sqrt{\lambda_2-nc^2})\label{ineqWangXia}
\end{align*}

Note that Xia had previously proved~\cite{xia}, under the same hypothesis, that
$\lambda_2 \ge nc^2$.

\medskip
In~\cite{ceg1}, we study the control of the spectrum of the Laplacian
on a closed hypersurface by the isoperimetric ratio. The following is
a particular case of one of our results.

\medskip
\begin{thm}\label{laplacienCEG}
  Let $N$ be a complete Riemannian manifold  with non-negative Ricci
  curvature. Let $\Omega \subset N$ be a bounded domain  with smooth
  boundary $\Sigma = \partial \Omega$ contained in a ball of radius
  $r<\frac{\mbox{inj}(N)}{2}$.
  There is a constant $B_n$ depending only on dimension such that for
  any $k\geq 0$,
  
  \begin{gather*}
    \bar\lambda_k(\Sigma)\leq B_n I(\Omega)^{(n+2)/n}{k}^{2/n}
  \end{gather*}
  where
  $\bar\lambda_k(\Sigma)=\lambda_k(\Sigma) {\vert  \Sigma \vert}^{2/n}$ 
  are the normalized eigenvalues of the Laplacian.
\end{thm}

\medskip
Combining Theorem~\ref{laplacienCEG} and Corollary~\ref{spaceforms},
we get

\begin{thm}\label{comparison}
  Let $N$ be a complete Riemannian manifold of dimension $(n+1)$ with
  non-negative Ricci curvature. There exists a constant $\kappa_n$
  depending only on dimension such that for any bounded domain
  $\Omega \subset N$ with boundary $\Sigma = \partial\Omega$ 
  contained in a ball of radius $r<\frac{\mbox{inj}(N)}{2}$ the
  following holds:
  
  \begin{gather}\label{comp1}
    \bar\lambda_k(\Sigma)\bar \sigma_l(\Omega)\leq
    \kappa_n\left(\frac{\vert \Sigma\vert}{\vert\Omega\vert}\right)^{3/n}k^{2/n}l^{2/(n+1)}.
  \end{gather}
  
  In the special case where $N$ is the Euclidean space 
  $\mathbb R^{n+1}$, the injectivity radius in each point is
  $\infty$, so that that is no further restrictions on $\Omega$,
  and Inequalities (\ref{comp1}) and (\ref{comp2}) are true for all bounded domains.
\end{thm}

 Without the normalization, we have
 \begin{gather}   \label{comp2}
   \lambda_k(\Sigma)\sigma_l(\Omega) m(\Sigma,\delta) \leq
   \kappa_n\frac{k^{2/n}l^{2/(n+1)}}{\vert\Omega\vert^{3/(n+1)}}.
 \end{gather}

\begin{rem}
  In comparison with~\cite{wangxia2}, we make no assumption on the
  convexity of $\Omega$. We also have comparison for all
  eigenvalues. Note however that our method does not give any
  sharpness.
\end{rem}

A remarkable feature of this inequality is that large eigenvalues of
the Laplacian are seen to impose small eigenvalues of the
Dirichlet-to-Neumann map.

\begin{cor}
  Under the assumptions of Theorem~\ref{comparison}, if a family of
  domains $\{\Omega_t\}_{0<t<1}$ of volume one with boundaries
  $\Sigma_t=\partial\Omega_t$ is such that
  $\displaystyle\lim_{t\rightarrow 0}\lambda_k(\Omega_t)=\infty,$
  then, if $n\geq 2$ we have for each $l\geq 1$
  $$\lim_{t\rightarrow 0}\bar\sigma_l(\Omega_t)\rightarrow 0.$$
\end{cor}

\section{Surfaces}
The situation for surface is special.
We begin by a proof of the upper bound of $\sigma_k$ in term of the
genus.
\begin{proof}[Proof of Theorem~\ref{surfaces}]
  This is a modification of the proof of Theorem~\ref{maindegree}.
  
  By gluing a disk on each boundary components of $M$, we can see $M$
  as a domain in a a compact surface $S$ of genus $\gamma$.
  This closed surface can be represented as a branched cover over
  $\mathbb{S}^2$ with degree $d=\lfloor\frac{\gamma+3}{2}\rfloor$
  (See~\cite{gunning} for instance).
  
  On $\mathbb{S}^2$ we consider the usual spherical distance $d$ and
  we define a Borel measure
  $\mu=\psi_* \left(\delta dv_{\Sigma}\right)$.
  That is, the measure of an open set $\mathcal{O}\subset\mathbb{S}^2$
  is 
  \begin{align}
    \mu(\mathcal{O})=\int_{\psi^{-1}(\mathcal{O})\cap\Sigma}\delta\,dv_{\Sigma}.
  \end{align}
  In particular,
  $$\mu(\mathbb{S}^2)=\vert \Sigma \vert m(\Sigma,\delta).$$
  
  It follows from Theorem~1 and Corollary~3.12 of~\cite{gyn} applied
  to the metric measured space $(\mathbb{S}^2,d,\mu)$ that there exist
  $2k$ annuli 
  $A_1,...,A_{2k} \subset \mathbb{S}^2$ with
  \begin{gather}\label{IneqMuAiboundedS}
    \mu(A_i) \ge \frac{\mu(\mathbb{S}^2)}{ck}.
  \end{gather}
  Because the annuli $2A_i$ are mutually disjoint, so are the sets 
  $B_i=\psi^{-1}(2A_i)$. These sets can be reordered so that the first
  $k$ of them satisfy 
  \begin{gather}\label{firstkS}
    \vert B_i \vert_g
    \leq\frac{\vert M \vert_g }{k}\ \ \ (i=1,\cdots,k).
  \end{gather}
  
  \medskip
  Let $A=\{x \in \mathbb{S}^2: r<d(x,a)<R\}$ be one of the above
  annuli and let $h$ be a function supported in $2A$. Let
  $f=h\circ\psi$ be 
  the lift of this function to $M$. It is supported in the set
  $B=\psi^{-1}(2A)$.
  
  Taking~(\ref{firstkS}) into account, it follows from conformal
  invariance of the Dirichlet energy that
  \begin{align*}
    \int_{B \cap \Omega} \vert \nabla_g f\vert^2\,dv_g
    &\leq
    \left(\mbox{deg}(\psi)\int_{2A} \vert \nabla_{g_0} h \vert^{n+1}\,dv_{g_0}\right)^{2/(n+1)}
    \left(\frac{\vert \Omega \vert_g}{k}\right)^{1-2/(n+1)}.
  \end{align*}
  
  The rest of the proof is almost identical to that of
  Theorem~\ref{maindegree} and is left to the reader.
\end{proof}

In Corollary~\ref{corLargeIso} it was mentioned that for manifold of
dimension at least three, a large isoperimetric ratio implies small
Steklov eigenvalues. The next example shows that this is false for
surfaces.
\begin{exa}\label{exampleSurface}
  Let $M$ be a compact Riemannian manifold with metric $g$.
  Let $f\in C^{\infty}(\bar M)$ be a smooth function vanishing on the
  boundary $\partial M$.  Consider a conformal perturbation
  $\tilde{g}=e^{f}g$ of the original metric. It is well known that the
  Laplacian is conformally invariant in dimension two. Moreover,
  because $\tilde{g}=g$ on $\partial M$, the normal derivative is also
  preserved. It follows that the the Dirichlet-to-Neumann map
  induced by $\tilde{g}$ is the same as that induced by by $g$. In
  particular, they have the same spectrum.

  \medskip
  On the other hand the measure of the surface is given by
  \begin{gather*}
    |M|_{\tilde{g}}=\int_Me^f\,dg.
  \end{gather*}
  By taking a function $f$ that decays fast away from the boundary, we
  can make this quantity as small as we want.
  In other words, the isoperimetric ratio
  $I(M)=\frac{|\partial M|}{\sqrt{|M|_g}}$ will become very large.
\end{exa}

\section{Construction of large eigenvalues}
\label{large}
The behavior of the Steklov spectrum depends on the interior of the
domain in an essential way. For a closed Riemannian manifold
$\Sigma$ with large eigenvalue $\lambda_k$ of the Laplacian, embedding
as an hypersurface in Euclidean space forces very small Steklov
eigenvalues.  This comes from the fact that, by~\cite{ceg1} the
isoperimetric ratio $I(\Omega)$ has to be big, with $\Sigma =\partial
\Omega$, and this implies the presence of small eigenvalues. If we
embed $\Sigma$ as the cross-section of a cylinder
$\Sigma\times\mathbb{R}$ with its product metric, we will see that
exactly the opposite will happen. This shows that our geometric
assumptions are necessary.

\medskip
\begin{lemma}\label{calculation}
  Let $\Sigma$ be a closed Riemannian manifold of volume one. Let the
  spectrum of its Laplace operator $\Delta_{\Sigma}$ be 
  $$0=\lambda_1<\lambda_2\leq\lambda_3\cdots\nearrow\infty$$
  and let $(u_k)$ be an orthonormal basis of $L^2(\Sigma)$ such that
  $$\Delta_{\Sigma} u_k=\lambda_ku_k.$$
  Let $N=\mathbb{R}\times\Sigma$.
  On the domain $\Omega=[-L,L]\times\Sigma \subset N$, a complete
  system of orthogonal eigenfunctions of the Dirichlet-to-Neumann map
  is given by  
  \begin{gather*}
    1,\ t,\
    \cosh(\sqrt{\lambda_k}t)f_k(x),\
    \sinh(\sqrt{\lambda_k}t)f_k(x)
  \end{gather*}
  with eigenvalues
  \begin{gather*}
    0, 1/L,\ 
    \sqrt{\lambda_k}\tanh (\sqrt{\lambda_k}L)<
    \sqrt{\lambda_k}\coth (\sqrt{\lambda_k}L).
  \end{gather*}
\end{lemma}
\begin{proof}
  It is enough to check that these functions are Steklov
  eigenfunctions since their restriction to the boundary form a basis
  $L^2$.
\end{proof}

\begin{prop}
  Let $\Sigma$ be a closed manifold of dimension $\geq 3$. On the
  product manifold $N=\Sigma\times\mathbb{R}$ there exists a complete
  Riemannian metric $g$ and a sequence of bounded domains $\Omega_i$
  such that
  \begin{gather*}
    \lim_{i\rightarrow\infty}\overline{\sigma}_2(\Omega_i)=\infty,
    \mbox{ and }\lim_{i\rightarrow\infty}I(\Omega_i)=\infty.
  \end{gather*}
\end{prop}

\begin{proof}
  Let $\Sigma$ be a closed manifold of dimension $\geq 3$. The first
  author and Dodziuk~\cite{cd} proved the existence of a sequence
  $h_i$ of Riemannian metrics of volume one such that
  $\lim_{i\rightarrow\infty}\lambda_2(\Sigma,h_i)=\infty.$
  Without loss of generality, we assume for each $i$ that
  $\lambda_2(\Sigma,h_i)>1.$
  
  \medskip
  Consider the cylinder
  $\Omega_i=\Sigma\times [i,i+L_i]$
  with
  $$1>L_i= \frac{1}{\sqrt{\lambda_2(\Sigma,h_i)}}\rightarrow 0\mbox{
    as } i\rightarrow\infty.$$
  
  \medskip
  Let $g$ be a complete Riemannian metric on $\Sigma\times\mathbb{R}$
  such that the restriction of $g$ to $\Omega_i$ is the product of
  $h_i$ with the Euclidean metric on $\mathbb{R}$.
  It follows from Lemma~\ref{calculation} that
  \begin{gather*}
    \sigma_2(\Omega_i)=
    \min\left(\sqrt{\lambda_2(\Sigma,h_i)},\sqrt{\lambda_2(\Sigma,h_i)}\tanh
      (1)\right)
    =\sqrt{\lambda_2(\Sigma,h_i)}\tanh (1).
  \end{gather*}
  
  In particular
  $$
  \lim_{i\rightarrow\infty}\bar\sigma_2(\Omega_i)=\infty.
  $$
\end{proof}

\begin{prop}
  There exists a complete three-dimensional Riemannian manifold $N$
  admitting a sequence of bounded domains $\Omega_i\subset N$ such
  that 
  \begin{gather*}
    \lim_{i\rightarrow\infty}\overline{\sigma}_2(\Omega_i)=\infty,
    \mbox{ and }\lim_{i\rightarrow\infty}I(\Omega_i)=\infty.
  \end{gather*}
\end{prop}

\begin{proof}
  It is well known that there exists a sequence of Riemann surfaces of
  volume one 
  $\Sigma_i$ such that $\lambda_2(\Sigma_i)\rightarrow\infty$
  (see~\cite{colboisElSoufi}).
  We consider  the complete Riemannian manifold
  $N_i= \Sigma_i\times \mathbb R$ (with the product Riemannian metric)
  and the subset    
  $$
  \Omega_i=\Sigma_i \times [0,L_i]
  $$  
  
  with $L_i= \frac{1}{\sqrt{\lambda_2(\Sigma_i)}}$.
  As before, we see that
  $\lim_{i\rightarrow\infty}\bar\sigma_2(\Omega_i) =\infty.$
  The manifold $N$ is obtained by joining the $N_i$'s by tubes.
\end{proof}

 \begin{prop}
   Let $M$ be a compact manifold of dimension $\geq 4$. There exists a
   sequence of Riemannian metric $g_i$ and a domain $\Omega\subset M$
   such that
  \begin{gather*}
      \lim_{i\rightarrow\infty}\overline{\sigma}_2(\Omega,g_i)=\infty,
      \mbox{ and }\lim_{i\rightarrow\infty}I(\Omega,g_i)=\infty.
    \end{gather*}
\end{prop}

\begin{proof}
  Let $\Omega$ be any domain of $M$ that is diffeomorphic to the
  cylinder $(0,1)\times\mathbb{S}^n$. Because $n\geq 3$, there exists
  a sequence of Riemannian metric $h_i$ on $\mathbb{S}^n$ such that
  $\lim_{i\rightarrow}\lambda_2(\mathbb{S}^n,h_i)=\infty$.
  Let $g_i$ be a Riemannian metric on $M$ such that the restriction of
  $g_i$ to $\Omega$ is isometric to product
  $\mathbb{S}^n\times(0,L_i)$ with
  $L_i= \frac{1}{\sqrt{\lambda_2(\mathbb{S}^n,h_i)}}$.
\end{proof}

It follows from scaling invariance of the normalized eigenvalues that
in each of the three previous examples, the Riemannian metrics can be
chosen to have Ricci curvature arbitrarily close to zero.

\bibliographystyle{bibliostyle}
\bibliography{biblioCEG}

\def\cprime{$'$} \def\cprime{$'$} \def\cprime{$'$}
\begin{thebibliography}{10}

\bibitem{band}
C.~Bandle, \textit{Isoperimetric inequalities and applications}, Monographs and
  Studies in Mathematics, vol.~7, Pitman (Advanced Publishing Program), Boston,
  Mass., 1980.

\bibitem{brock}
F.~Brock, \textit{An isoperimetric inequality for eigenvalues of the {S}tekloff
  problem}, ZAMM Z. Angew. Math. Mech., \textbf{81} (2001), 69--71.

\bibitem{cldr}
A.-P. Calder{\'o}n, \textit{On an inverse boundary value problem}, dans Seminar
  on {N}umerical {A}nalysis and its {A}pplications to {C}ontinuum {P}hysics
  ({R}io de {J}aneiro, 1980), 65--73, Soc. Brasil. Mat., Rio de Janeiro, 1980.

\bibitem{cha2}
I.~Chavel, \textit{Isoperimetric inequalities}, Cambridge Tracts in
  Mathematics, vol. 145, Cambridge University Press, Cambridge, 2001.

\bibitem{cd}
B.~Colbois et J.~Dodziuk, \textit{Riemannian metrics with large {$\lambda_1$}},
  Proc. Amer. Math. Soc., \textbf{122} (1994), 905--906.

\bibitem{colboisElSoufi}
B.~Colbois et A.~El~Soufi, \textit{Extremal eigenvalues of the {L}aplacian in a
  conformal class of metrics: the `conformal spectrum'}, Ann. Global Anal.
  Geom., \textbf{24} (2003), 337--349.

\bibitem{ceg1}
B.~Colbois, A.~El~Soufi et A.~Girouard, \textit{Isoperimetric control of the
  spectrum of hypersurfaces}, to appear,  (2010).

\bibitem{croke}
C.~B. Croke, \textit{Some isoperimetric inequalities and eigenvalue estimates},
  Ann. Sci. \'Ecole Norm. Sup. (4), \textbf{13} (1980), 419--435.

\bibitem{esco1}
J.~F. Escobar, \textit{The geometry of the first non-zero {S}tekloff
  eigenvalue}, J. Funct. Anal., \textbf{150} (1997), 544--556.

\bibitem{esco2}
J.~F. Escobar, \textit{An isoperimetric inequality and the first {S}teklov
  eigenvalue}, J. Funct. Anal., \textbf{165} (1999), 101--116.

\bibitem{esco3}
J.~F. Escobar, \textit{A comparison theorem for the first non-zero {S}teklov
  eigenvalue}, J. Funct. Anal., \textbf{178} (2000), 143--155.

\bibitem{foxkut}
D.~W. Fox et J.~R. Kuttler, \textit{Sloshing frequencies}, Z. Angew. Math.
  Phys., \textbf{34} (1983), 668--696.

\bibitem{fraschoen}
A.~Fraser et R.~Schoen, \textit{The first {S}teklov eigenvalue, conformal
  geometry, and minimal surfaces}, to appear in Advances in Math,  (2011).

\bibitem{gp}
A.~Girouard et I.~Polterovich, \textit{On the {H}ersch-{P}ayne-{S}chiffer
  estimates for the eigenvalues of the {S}teklov problem}, Funktsional. Anal. i
  Prilozhen., \textbf{44} (2010), 33--47.

\bibitem{gp2}
A.~Girouard et I.~Polterovich, \textit{Shape optimization for low {N}eumann and
  {S}teklov eigenvalues}, Math. Methods Appl. Sci., \textbf{33} (2010),
  501--516.

\bibitem{gyn}
A.~Grigor{\cprime}yan, Y.~Netrusov et S.-T. Yau, \textit{Eigenvalues of
  elliptic operators and geometric applications}, dans Surveys in differential
  geometry. {V}ol. {IX}, Surv. Differ. Geom., IX, 147--217, Int. Press,
  Somerville, MA, 2004.

\bibitem{gunning}
R.~C. Gunning, \textit{Lectures on {R}iemann surfaces, {J}acobi varieties},
  Princeton University Press, Princeton, N.J., 1972.

\bibitem{hps}
J.~Hersch, L.~E. Payne et M.~M. Schiffer, \textit{Some inequalities for
  {S}tekloff eigenvalues}, Arch. Rational Mech. Anal., \textbf{57} (1975),
  99--114.

\bibitem{kokrein}
N.~D. Kopachevsky et S.~G. Krein, \textit{Operator approach to linear problems
  of hydrodynamics. {V}ol. 1}, Operator Theory: Advances and Applications, vol.
  128, Birkh\"auser Verlag, Basel, 2001.

\bibitem{kvr}
N.~Korevaar, \textit{Upper bounds for eigenvalues of conformal metrics}, J.
  Differential Geom., \textbf{37} (1993), 73--93.

\bibitem{LaTaUh}
M.~Lassas, M.~Taylor et G.~Uhlmann, \textit{The {D}irichlet-to-{N}eumann map
  for complete {R}iemannian manifolds with boundary}, Comm. Anal. Geom.,
  \textbf{11} (2003), 207--221.

\bibitem{payne}
L.~E. Payne, \textit{Some isoperimetric inequalities for harmonic functions},
  SIAM J. Math. Anal., \textbf{1} (1970), 354--359.

\bibitem{shamma}
S.~E. Shamma, \textit{Asymptotic behavior of {S}tekloff eigenvalues and
  eigenfunctions}, SIAM J. Appl. Math., \textbf{20} (1971), 482--490.

\bibitem{stek}
W.~Stekloff, \textit{Sur les probl\`emes fondamentaux de la physique
  math\'ematique}, Ann. Sci. \'Ecole Norm. Sup. (3), \textbf{19} (1902),
  191--259.

\bibitem{taylorPDEII}
M.~E. Taylor, \textit{Partial differential equations. {II}}, Applied
  Mathematical Sciences, vol. 116, Springer-Verlag, New York, 1996.

\bibitem{wangxia2}
Q.~Wang et C.~Xia, \textit{Sharp bounds for the first non-zero {S}tekloff
  eigenvalues}, J. Funct. Anal., \textbf{257} (2009), 2635--2644.

\bibitem{wein}
R.~Weinstock, \textit{Inequalities for a classical eigenvalue problem}, J.
  Rational Mech. Anal., \textbf{3} (1954), 745--753.

\bibitem{xia}
C.~Xia, \textit{Rigidity of compact manifolds with boundary and nonnegative
  {R}icci curvature}, Proc. Amer. Math. Soc., \textbf{125} (1997), 1801--1806.

\end{thebibliography}

\end{document}